[1994/12/01]
\documentclass{ijmart-mod}
\chardef\bslash=`\\ 

\usepackage{graphicx}
\usepackage[breaklinks=true]{hyperref}

\newtheorem{thm}{Theorem}[section]

\newtheorem{lem}[thm]{Lemma}
\newtheorem{prop}[thm]{Proposition}

\theoremstyle{definition}
\newtheorem{defn}[thm]{Definition}

\newtheorem{qtn}[thm]{Question}

\theoremstyle{remark}

\newcommand{\eval}[2][\right]{\relax
  \ifx#1\right\relax \left.\fi#2#1\rvert}

\begin{document}
\title[Diameter, decomposability, and Minkowski sums]{Diameter, decomposability, and Minkowski~sums~of~polytopes}

\author[Antoine Deza]{Antoine Deza}
\address{McMaster University, Hamilton, Ontario, Canada}
\email{deza@mcmaster.ca} 

\author[Lionel Pournin]{Lionel Pournin}
\address{LIPN, Universit{\'e} Paris 13, Villetaneuse, France}
\email{lionel.pournin@univ-paris13.fr} 

\begin{abstract}
We investigate how the Minkowski sum of two polytopes affects their graph and, in particular, their diameter. We show that the diameter of the Minkowski sum is bounded below by the diameter of each summand and above by, roughly, the product between the diameter of one summand and the number of vertices of the other. We also prove that both bounds are sharp. In addition, 
we obtain a result on polytope decomposability. More precisely,
given two polytopes $P$ and $Q$, we show that $P$ can be written as a Minkowski sum with a summand homothetic to $Q$ if and only if $P$ has the same number of vertices as its Minkowski sum with $Q$.
\end{abstract}
\maketitle

\section{Introduction}\label{sec.DFP.0}
The Minkowski sum of two subsets of an Euclidean space is obtained by summing each element of one subset with each element of the other. The Minkowski sum of $P$ and $Q$ is denoted by $P+Q$. This operation turns up in a large number of different contexts ranging from the Brunn-Minkowski theorem to applications in civil engineering or motion planning. The special case when $P$ and $Q$ are polytopes is of particular interest. It is a model for the combinatorics of prismatoids used by Santos to disprove the Hirsch conjecture \cite{Santos2012}. The face lattice of $P+Q$, and in particular its vertex set, has been studied by Fukuda and Weibel~\cite{FukudaWeibel2007}. Recently, a sharp upper bound on the number of faces of $P+Q$ has been obtained by Adiprasito and Sanyal \cite{AdiprasitoSanyal2016}. The question of the decomposability of a polytope, that is, whether it can be obtained as the Minkowski sum of two non-homothetic polytopes has been considered in~\cite{Kallay1982,Meyer1974,PrzeslawskiYost2008,Shephard1963}. Among polytopes, the case of zonotopes is particularly interesting. These polytopes are the Minkowski sums of line segments. Zonotopes are conjectured, for any pair of positive integers $d$ and $k$, to achieve the largest possible diameter over all the $d$-dimensional polytopes whose vertices have integer coordinates ranging from $0$ to $k$ \cite{DezaManoussakisOnn2018}. Here, by the \emph{diameter of a polytope}, we mean the diameter of the graph of a polytope, made up of its vertices and edges.
We refer to the textbooks by Fukuda~\cite{Fukuda2015}, Gr\"unbaum~\cite{Grunbaum2003}, and Ziegler~\cite{Ziegler1995} 
for comprehensive introductions on polytopes, Minkowski sums, and zonotopes. 

Here, we focus on the possible diameter of (the graph of) the Minkowski sum of two polytopes. While this diameter is bounded below by the diameters of each summand, we will observe that it can grow arbitrarily large even when the diameter of both summands is fixed. In fact, we will prove that this diameter cannot exceed, roughly, the product between the diameter of one summand and the number of vertices of the other. We will also show that this upper bound is sharp when the diameter and the number of vertices of both summands grow large. Along the way, we obtain a result on the decomposability of a polytope into a Minkowski sum. If $P$ is the Minkowski sum of two polytopes $Q$ and $R$, we say that $Q$ and $R$ are \emph{summands} of $P$. A polytope that is not homothetic to at least one of its summands is called \emph{decomposable} \cite{Shephard1963}. We will show that a polytope $P$ has a summand homothetic to a polytope $Q$ if and only if $P$ and $P+Q$ have the same number of vertices. 
This allows for a convenient way to check polytope decomposability, especially in the case of lattice polytopes.

The article is based on a couple of propositions from \cite{Fukuda2015}, which we recall and extend in Section \ref{sec.DFP.1}. Our result on polytope decomposability is given as a conclusion to Section~\ref{sec.DFP.1}. The question on the diameter of Minkowski sums is addressed in Sections \ref{sec.DFP.2} and \ref{sec.DFP.4}. The bounds on that diameter are given in Section \ref{sec.DFP.2} and the proof that the upper bound is sharp in Section \ref{sec.DFP.4}.

\section{Some properties of the Minkowski sum of polytopes}\label{sec.DFP.1}

In the following, each time a Minkowski sum of two polytopes is considered, it is implicitly assumed that these polytopes are both contained in the same ambient Euclidean space. Note that we will make heavy use of linear maps of the form $x\mapsto{c\mathord{\cdot}x}$. 
In this notation, $c$ and $x$ are vectors in the considered ambient space and $c\mathord{\cdot}x$ denotes their scalar product.

The following Lemma is borrowed from \cite{Fukuda2015}. It is in some sense our starting point. In particular, most of our results are based on it.

\begin{lem}[{\cite[Proposition 12.1]{Fukuda2015}}]\label{Lem.DFP.1.1}
For any subset $F$ of a polytope $P$ and any subset $G$ of a polytope $Q$, $F+G$ is a face of $P+Q$ if and only if
\begin{itemize}
\item[$(i)$] $F$ and $G$ are faces of $P$ and $Q$, respectively,
\item[$(ii)$] there exists a vector $c$ such that the map $x\mapsto{c\mathord{\cdot}x}$ is minimized exactly at $F$ in $P$ and exactly at $G$ in $Q$.
\end{itemize}
\end{lem}

By this lemma, given two polytopes $P$ and $Q$, a face $X$ of their Minkowski sum can always be written as the Minkowski sum of a unique face $F$ of $P$ and a unique face $G$ of $Q$. In the sequel, the expression $F+G$ will be referred to as the \emph{Minkowski decomposition} of $X$.
\begin{figure}
\begin{centering}
\includegraphics{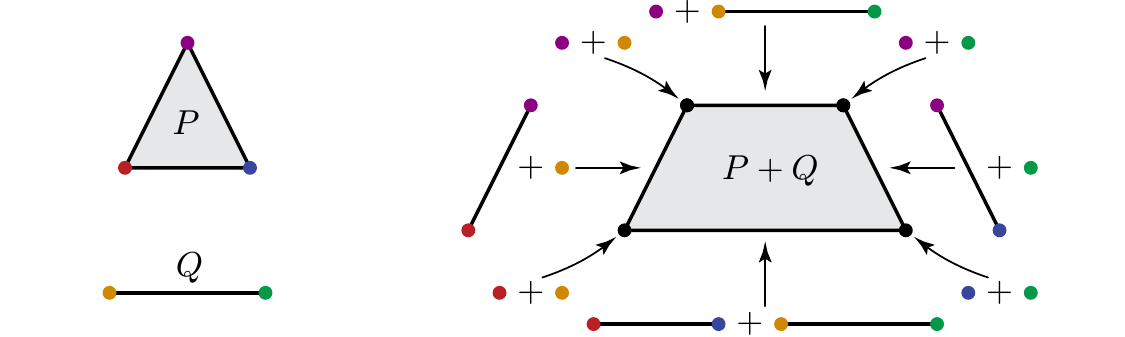}
\caption{The Minkowski sum of a triangle and a line segment.}\label{Fig.DFP.0}
\end{centering}
\end{figure}
Lemma \ref{Lem.DFP.1.1} is illustrated on Fig. \ref{Fig.DFP.0} with the Minkowski sum of a triangle $P$ and a line segment $Q$, where the Minkowski decomposition of each proper face of $P+Q$ is indicated by an arrow. Note, for instance that, when $c$ is a vertical vector pointing down, the map $x\mapsto{c\mathord{\cdot}x}$ is minimized, in $P$, at the purple vertex and, in $Q$, at $Q$ itself. The sum of these two faces is the line segment at the top of $P+Q$. The following lemma, also borrowed from \cite{Fukuda2015} tells how Minkowski sums affect vertex adjacency.

\begin{lem}[{\cite[Proposition 12.4]{Fukuda2015}}]\label{Lem.DFP.1.2}
Let $P$ and $Q$ be two polytopes. If $u$ and $v$ are adjacent vertices of $P+Q$ with Minkowski decompositions $u_P+u_Q$ and $v_P+v_Q$, respectively, then $u_P$ and $v_P$ are either adjacent vertices of $P$, or they coincide. Similarly, $u_Q$ and $v_Q$ are adjacent vertices of $Q$, or they coincide.
\end{lem}

Observe that, for any vertex $u$ of a polytope $P$, and any polytope $Q$, there exists a vertex $v$ of $Q$ such that $u+v$ is a vertex of $P+Q$. Indeed, consider a vector $c$ such that the map $x\mapsto{c\mathord{\cdot}x}$ is uniquely minimized at $u$ in $P$. This map is also minimized at a face $F$ in $Q$. According to Lemma \ref{Lem.DFP.1.1}, $u+F$ is a face of $P+Q$, and the vertices of this face are precisely the Minkowski sums $u+v$ where $v$ is a vertex of $F$. Since the Minkowski decomposition of a vertex of $P+Q$ is unique, we immediately obtain Lemma \ref{Lem.DFP.1.3}.

\begin{lem}\label{Lem.DFP.1.3}
Let $P$ and $Q$ be two polytopes. There exists an injection $\phi$ from the vertex set of $P$ into the vertex set of $P+Q$ such that, for every vertex $u$ of $P$, $\phi(u)=u+v$, where $v$ is a vertex of $Q$.
\end{lem}

Consider a face $F$ of a polytope $P$. Recall that the normal cone of $P$ at $F$ is the set of all the vectors $c$ such that the map $x\mapsto{c\mathord{\cdot}x}$ is minimized, in $P$, at a face that contains $F$. The normal fan of $P$ is the complete polyhedral fan made up of the normal cones of $P$ at all of its faces. When the injection provided by Lemma \ref{Lem.DFP.5.1} between the vertex sets of $P$ and $P+Q$ is a bijection, we will show that their normal fans coincide. Note that, as an immediate consequence, the face lattices of these two polytopes are isomorphic.

\begin{lem}\label{Lem.DFP.5.1}
Consider two polytopes $P$ and $Q$. Let $\phi$ be an injection from the vertex set of $P$ to the vertex set of $P+Q$ such that, for every vertex $u$ of $P$, $\phi(u)=u+v$, where $v$ is a vertex of $Q$. If $\phi$ is a bijection, then the normal fan of $P$ coincides with the normal fan of $P+Q$.
\end{lem}
\begin{proof}
By Proposition 7.12 from \cite{Ziegler1995}, the normal fan of $P+Q$ refines the normal fan of $P$. In other words, the normal cones of $P+Q$ form polyhedral subdivisions of each of the normal cones of $P$. Hence, in order to prove the lemma, it suffices to exhibit a bijection between the two normal fans.

Consider a proper face $F$ of $P$. By Lemma \ref{Lem.DFP.1.1}, there exists a face $\psi(F)$ of $Q$ such that $F+\psi(F)$ is a face of $P+Q$. We first show that $\phi$ takes the vertex set of $F$ to the vertex set of $F+\psi(F)$. Let $u$ be a vertex of $F$. By Lemma \ref{Lem.DFP.1.3}, there exists a vertex $v$ of $\psi(F)$ such that $u+v$ is a vertex of $F+\psi(F)$. Since $\phi$ is a bijection from the vertex set of $P$ to the vertex set of $P+Q$, $u+v$ admits an antecedent by $\phi$ and this antecedent is, by definition, $u$ itself. Hence $\phi(u)$ is indeed a vertex of $F+\psi(F)$. As any vertex of $F+\psi(F)$ is obtained as the Minkowski sum of a vertex of $F$ with a vertex of $\psi(F)$, this shows that $\phi$ takes the vertex set of $F$ precisely to the vertex set of $F+\psi(F)$. As a consequence, $\psi(F)$ is the only possible face of $Q$ such that $F+\psi(F)$ is a face of $P+Q$. Since every face of $P+Q$ is the Minkowski sum of a face of $P$ and a face of $Q$, the map $F\mapsto{F+\psi(F)}$ is a one to one correspondence between the proper faces of $P$ and the proper faces of $P+Q$. By duality, there is a bijection between the normal fan of $P$ and the normal fan of $P+Q$.
\end{proof}

Lemma \ref{Lem.DFP.5.1} concludes to the identity of the normal fans of two polytopes. According to the following result, proven in \cite{Kallay1982}, this situation has a particular meaning in terms of the summands of these polytopes.

\begin{lem}[{\cite[Theorem 4]{Kallay1982}}]\label{Lem.DFP.5.2}
If the normal fans of two polytopes $P$ and $Q$ coincide, then $P$ has a summand homothetic to $Q$.
\end{lem}

Note that Theorem 4 from \cite{Kallay1982} actually provides four statements equivalent to the normal fans of two polytopes coinciding. Lemma \ref{Lem.DFP.5.2} only borrows the part of this theorem that we will make use of here.

\begin{thm}\label{Thm.DFP.5.1}
A polytope $P$ has a summand homothetic to a polytope $Q$ if and only if $P$ and $P+Q$ have the same number of vertices.
\end{thm}
\begin{proof}
Assume that $P$ has a summand homothetic to $Q$, that is $P=\alpha{Q}+R$ for some positive number $\alpha$ and some polytope $R$. In this case, Lemma \ref{Lem.DFP.1.1} provides a bijection between the vertex set of $P$ and the vertex set of $P+Q$. Indeed, let $u$ and $v$ be two points in $Q$ and $R$, respectively. By Lemma \ref{Lem.DFP.1.1}, $\alpha{u}+v$ is a vertex of $P$ if and only if there exists a vector $c$ such that the map $x\mapsto{c\mathord{\cdot}x}$ is uniquely minimized at $\alpha{u}$ in $\alpha{Q}$ and at $v$ in $R$. This is equivalent to the map $x\mapsto{c\mathord{\cdot}x}$ being uniquely minimized at $(1+\alpha)u$ in $(1+\alpha)Q$ and at $v$ in $R$. Since $P+Q=(1+\alpha)Q+R$, it follows from Lemma \ref{Lem.DFP.1.1} that the map $\alpha{u}+v\mapsto(1+\alpha)u+v$ is a bijection between the vertices of $P$ and $P+Q$.

Now assume that $P$ and $P+Q$ have the same number of vertices. In this case, the injection provided by Lemma \ref{Lem.DFP.1.3} is a bijection. According to Lemma \ref{Lem.DFP.5.1}, the normal fans of $P$ and $P+Q$ then coincide and, in turn, by Lemma \ref{Lem.DFP.5.2}, $P$ has a summand homothetic to $P+Q$. As a direct consequence, $P$ has a summand homothetic to $Q$, and the proof is complete.
\end{proof}

A weaker version of Theorem \ref{Thm.DFP.5.1} where $P$ is a lattice polytope and $Q$ is lattice segment is used in \cite{DezaDezaGuanPournin2018} in order to enumerate lattice polytopes with given properties. Note that, in the case of lattice polytopes, the summand homothetic to $Q$ in the statement of Theorem \ref{Thm.DFP.5.1} is necessarily homothetic to $Q$ by an integer coefficient, which allows for an convenient enumeration procedure. A consequence of Theorem \ref{Thm.DFP.5.1} is that it makes it possible to check whether a Minkowski difference is possible between $P$ and a polytope homothetic to $Q$ by only computing the vertices $P+Q$ and comparing its number of vertices to that of $P$. Another consequence is that it provides an efficient way to tell whether a lattice polytope $P$ is a zonotope: it suffices to compute the Minkowski sum of $P$ with each of its edges (up to parallelism) and, for each of them, to compare the number of vertices of the resulting polytope with that of $P$. 

\section{Bounds on the diameter of Minkowski sums}\label{sec.DFP.2}

The purpose of this section is to investigate the possible range for the diameter of a Minkowski sum in terms of the diameter and the number of vertices of its summands. In the remainder of the article, the diameter of a polytope $P$ will be denoted by $\delta(P)$. We begin with a general lower bound that only depends on the diameter of the summands.

\begin{thm}\label{Thm.DFP.2.2}
For any two polytopes $P$ and $Q$,
$$
\delta(P+Q)\geq\max\{\delta(P),\delta(Q)\}\mbox{.}
$$
\end{thm}
\begin{proof}
By Lemma \ref{Lem.DFP.1.3}, there exists an injection $\phi$ from the vertex set of $P$ into the vertex set of $P+Q$ such that, for every vertex $v$ of $P$, the Minkowski decomposition of $\phi(v)$ contains $v$ as one of its two summands. Consider two vertices $u$ and $v$ of $P$ distance of $\delta(P)$ in the graph of $P$. By Lemma \ref{Lem.DFP.1.2}, for any path of length $l$ between $\phi(u)$ and $\phi(v)$ in the graph of $P+Q$, there exists a path of length at most $l$ between $u$ and $v$ in the graph of $P$. As a consequence, the distance between $u$ and $v$ in the graph of $P$ is at most the distance between $\phi(u)$ and $\phi(v)$ in the graph of $P+Q$. Therefore, $\delta(P)\leq\delta(P+Q)$ and, by symmetry, the desired inequality holds.
\end{proof}

The inequality provided by Theorem~\ref{Thm.DFP.2.2} is sharp since $\delta(2P)=\delta(P)$ for any polytope $P$. This inequality is used in \cite{DezaDezaGuanPournin2018} in the case when $Q$ is a line segment, in order to evaluate the diameter of lattice polytopes.

It turns out that there is no upper bound on the diameter of a Minkowski sum only in terms of the diameter of the summands. More precisely we provide a pair of polytopes, each of diameter $2$, whose diameter of the Minkowski sum can grow arbitrarily large. The construction of these polytopes relies on the following proposition, that provides polytopes of any dimension and any diameter.

\begin{prop}\label{Prop.DFP.2.1}
For any two positive integers $d$ and $k$, there exists a polytope of dimension $d$ and diameter $k$.
\end{prop}
\begin{proof}
We shall distinguish two cases. First assume that $k\geq{d-1}$. Consider a polygon with $2(k-d)+5$ vertices (whose diameter is therefore $k-d+2$) and a $(d-2)$-dimensional cube. Let $P$ be the cartesian product of the polygon and the cube. This cartesian product can be alternatively obtained by taking a prism over the polygon, and then a prism over this prism, and so on until the resulting polytope is $d$-dimensional. Since the diameter of a prism is the diameter of its base plus $1$, the diameter of $P$ is equal to $k$.

Now assume that $k<d-1$. Consider a $(d-k+1)$-dimensional simplex and a $(k-1)$-dimensional cube.  As above, the Minkowski sum $P$ of the simplex and the cube is a $d$-dimensional polytope obtained by taking successive prisms over the simplex. Therefore, as the diameter of a prism is the diameter of its base plus $1$ and as simplices have diameter $1$, the diameter of $P$ is equal to $k$.
\end{proof}

By Proposition \ref{Prop.DFP.2.2}, the diameter of a Minkowski sum of two polytopes can grow arbitrarily large, even if both polytopes have a fixed diameter.

\begin{prop}\label{Prop.DFP.2.2}
For any $d\geq3$ and $k\geq4$, there exist two $d$-dimensional polytopes, both of diameter $2$, whose Minkowski sum has diameter $k$
\end{prop}

\begin{proof}
By Proposition \ref{Prop.DFP.2.1}, there exists a polytope $B$ of dimension $d-1$ and diameter $k-2$. We will think of $B$ as embedded in a hyperplane $H$ of $\mathbb{R}^d$. Consider two points $p$ and $q$ placed in $\mathbb{R}^d\mathord{\setminus}H$ in such a way that the line segment between $p$ and $q$ goes through the relative interior of $B$. Let $P$ and $Q$ be the pyramids over $B$ whose apices are $p$ and $q$. By construction, $P$ and $Q$ both have diameter $2$. Note that $p+B$ and $q+B$ are two translates of $B$ placed in distinct hyperplanes parallel to $H$. The Minkowski sum of $P$ and $Q$ is the convex hull of these two translates of $B$, and of the polytope $2B$ (the Minkowski sum of $B$ with itself) placed between them in a third hyperplane parallel to $H$. In particular all the faces of $p+B$ and $q+B$ are also faces of $P+Q$. Moreover, since the line segment between $p$ and $q$ goes through the relative interior of $B$, all the proper faces of $2B$ are faces of $P+Q$, and all the remaining faces of $P+Q$ are precisely obtained as the convex hull of $x+F$ and $2F$, where $F$ is a face of $B$, and $x$ is equal to $p$ or to $q$. Combinatorially, $P+Q$ can be thought of as a prism on both sides of $B$. Since the diameter of a prism is the diameter of its base plus $1$, the diameter of $P+Q$ is equal to $k$.
\end{proof}

When $d$ is equal to $3$, the construction in the proof of Proposition \ref{Prop.DFP.2.2} consists in considering a convex polygon $B$ with $2k-3$ vertices and two pyramids $P$ and $Q$ over this polygon whose apices are joined by a line segment going through the relative interior of $B$. A property of this construction is that both $P$ and $Q$ have diameter $2$. It would be interesting to know whether a statement similar to that of Proposition \ref{Prop.DFP.2.2} is true with polytopes of smaller diameter.

\begin{qtn}
Does there exist a polytope of diameter $1$ and a polytope of diameter $1$ or $2$ whose Minkowski sum is arbitrarily large?
\end{qtn}

On the one hand, Proposition \ref{Prop.DFP.2.2} shows that there is no finite upper bound on the diameter of a Minkowski sum of polytopes only in terms of the diameter of the summands. In other words, the ratio
$$
\frac{\delta(P+Q)}{\delta(P)\delta(Q)}
$$
can grow arbitrarily large. On the other hand, there is a coarse upper bound for the diameter of $P+Q$ in terms of the number of vertices of $P$ and $Q$, which we denote by $f_0(P)$ and $f_0(Q)$, respectively. Since a geodesic in the graph of $P+Q$ cannot visit a vertex twice, the diameter of $P+Q$ is at most the number of vertices of $P+Q$, which is in turn bounded above by $f_0(P)f_0(Q)$. The main result of this section is the following refined bound, that combines the diameters of $P$ and $Q$ and the number of their vertices.
\begin{thm}\label{Thm.DFP.3.1}
For any two polytopes $P$ and $Q$,
$$
\delta(P+Q)<\min\left\{(\delta(P)+1)f_0(Q),f_0(P)(\delta(Q)+1)\right\}\mbox{.}
$$
\end{thm}

As will be shown in Section~\ref{sec.DFP.4}, this bound is sharp when the diameter of one summand grows large and the other summand is a line segment or a polygon with an arbitrarily large number of vertices. In order to prove Theorem \ref{Thm.DFP.3.1}, we introduce the following family of graphs, whose vertex sets form a partition of the vertices of the Minkowski sum.

\begin{defn}\label{Def.DFP.3.1}
Consider two polytopes $P$ and $Q$. For any vertex $u$ of $P$, call $\Gamma_{P,Q}(u)$ the subgraph induced in the graph of $P+Q$ by the vertices whose Minkowski decomposition is of the form $u+v$, where $v$ is a vertex of $Q$.
\end{defn}

Note that the injection $\phi$ provided by Lemma \ref{Lem.DFP.1.3} is precisely a map that sends each vertex $u$ of $P$ to a vertex of $\Gamma_{P,Q}(u)$. Let us illustrate graphs $\Gamma_{P,Q}(u)$ using the Minkowski sum of a triangle $P$ and a line segment $Q$ depicted in Fig. \ref{Fig.DFP.0}. One can see on the right of the figure that, when $u$ is the purple vertex of $P$, $\Gamma_{P,Q}(u)$ is the graph made up of the line segment at the top of $P+Q$ and its two vertices. When $u$ is the red or the blue vertex of $P$, $\Gamma_{P,Q}(u)$ is made up of a single vertex and no edge; this vertex is the one bottom left of $P+Q$ if $u$ is the red vertex of $P$, and bottom right of $P+Q$ if $u$ is the blue vertex of $P$. Further observe that $\Gamma_{Q,P}(u)$ is the oblique edge on the left of $P+Q$ together with its vertices when $u$ is the yellow vertex of $Q$ and the other oblique edge of $P+Q$ together with its vertices when $u$ is the green vertex of $Q$.

\begin{lem}\label{Lem.DFP.3.1}
Consider two polytopes $P$ and $Q$. For any vertex $u$ of $P$, the graph $\Gamma_{P,Q}(u)$ is connected.
\end{lem}
\begin{proof}
Consider the normal cone $N$ of $P$ at $u$. 
By Lemma \ref{Lem.DFP.1.1} the Minkowski sum of $u$ with a face $F$ of $Q$ is a face of $P+Q$ if and only if the normal cone of $Q$ at $F$ is non-disjoint from $N$. Therefore, by Definition \ref{Def.DFP.3.1}, $u+v$ is a vertex of $\Gamma_{P,Q}(u)$ if and only if the normal cone of $Q$ at $v$ is non-disjoint from $N$. Let $v$ and $w$ be two vertices of $Q$ such that $u+v$ and $u+w$ are vertices of $\Gamma_{P,Q}(u)$. Choose a point $p_v$ in the intersection of $N$ and the normal cone of $Q$ at $v$. Similarly, let $p_w$ be a point in the intersection of $N$ and the normal cone of $Q$ at $w$. Since the normal cone of a polytope at a vertex is open and full dimensional, we can assume that the line segment between $p_v$ and $p_w$ does not meet a face of dimension less than $d-1$ in the normal fan of $Q$. This can be achieved by, if needed, perturbing $p_v$ slightly. By construction, when going from $p_v$ to $p_w$ along the line segment that joins these points, one meets a sequence of full-dimensional cones in the normal fan of $Q$, glued along cones of codimension $1$. These cones are the normal cones of $Q$ at the vertices and the edges of a path in the graph of $Q$ from $v$ to $w$. By the convexity of $N$, all of these cones are non-disjoint from $N$. By the above observation, the Minkowski sum of $u$ with the vertices and the edges of the path we found in the graph of $Q$ from $v$ to $w$ form a path from $u+v$ to $u+w$ in $\Gamma_{P,Q}(u)$.
\end{proof}

Lemma \ref{Lem.DFP.3.2} tells how the subgraphs induced by the graphs $\Gamma_{P,Q}(u)$ relate to one another within the graph of $P+Q$.

\begin{lem}\label{Lem.DFP.3.2}
Consider two polytopes $P$ and $Q$. Two distinct vertices $u$ and $v$ of $P$ are adjacent in the graph of $P$ if and only if there exist a vertex of $\Gamma_{P,Q}(u)$ and a vertex of $\Gamma_{P,Q}(v)$ that are adjacent in the graph of $P+Q$.
\end{lem}
\begin{proof}
First consider an edge of $P+Q$ between a vertex of $\Gamma_{P,Q}(u)$ and a vertex of $\Gamma_{P,Q}(v)$. This edge is the Minkowski sum of a face of $P$ with a face of $Q$, both of dimension $0$ or $1$. It turns out that the face of $P$ is necessarily the line segment with vertices $u$ and $v$, because $u$ and $v$ are distinct.

Now assume that $u$ and $v$ are adjacent in the graph of $P$. Consider a projection $\pi$ on some linear hyperplane $H$ of the ambient space that sends $u$ and $v$ to the same point. Observe that $\pi(u)$ is a vertex of $\pi(P)$ and consider a vector $c\in\mathbb{R}^d$ such that the map $x\mapsto{c\mathord{\cdot}x}$ is uniquely minimized at $\pi(u)$ in $\pi(P)$. This map is minimized at a face $F$ in $Q$. According to Lemma \ref{Lem.DFP.1.1}, $\pi(u)+F$ is a face of $\pi(P)+\pi(Q)$, and the vertices of this face are precisely the Minkowski sums of $u$ with the vertices of $F$. Hence there exists a vertex of $\pi(P)+\pi(Q)$, obtained as the Minkowski sum of $\pi(u)$ with a vertex, say $\pi(w)$ of $\pi(Q)$. Since Minkowski sums commute with projections, $\pi(u+w)$ is a vertex of $\pi(P+Q)$. Now observe that the face of $P+Q$ whose image by $\pi$ is $u+w$ is either a vertex or an edge. Since $u$ and $v$ are distinct, this face is an edge between a vertex of $\Gamma_{P,Q}(u)$ and a vertex of $\Gamma_{P,Q}(v)$.
\end{proof}

We are now ready to prove Theorem \ref{Thm.DFP.3.1}.

\begin{proof}[Proof of Theorem \ref{Thm.DFP.3.1}]
Consider two vertices of $u$ and $v$ of $P$ such that the largest possible distance,  in the graph of $P+Q$, between a vertex of $\Gamma_{P,Q}(u)$ and a vertex of $\Gamma_{P,Q}(v)$ is exactly $\delta(P+Q)$. Denote by $l$ the distance of $u$ and $v$ in the graph of $P$. We are going to show that the distance, in the graph of $P+Q$, between any vertex of $\Gamma_{P,Q}(u)$ and any vertex of $\Gamma_{P,Q}(v)$, that is $\delta(P+Q)$, is at most $(l+1)f_0(Q)$. Consider a geodesic from $u$ to $v$ in the graph of $P$. Denote by $w^0$ to $w^l$ the vertices along this geodesic in such a way that $w^0=u$, $w^l=v$, and $w^{i-1}$ is adjacent to $w^i$ in the graph of $P$ for all $i$.

According to Lemma \ref{Lem.DFP.3.1}, $\Gamma_{P,Q}(w^i)$ is a connected graph. We will denote the diameter of this graph by $\delta(\Gamma_{P,Q}(w^i))$. By Lemma \ref{Lem.DFP.3.2}, some vertex of $\Gamma_{P,Q}(w^{i-1})$ is adjacent to a vertex of $\Gamma_{P,Q}(w^i)$ in the graph of $P+Q$. Therefore, the largest distance in the graph of $P+Q$ between any vertex of $\Gamma_{P,Q}(u)$ and any vertex of $\Gamma_{P,Q}(v)$, that is the diameter of $P+Q$, is bounded as follows:
\begin{equation}\label{Thm.DFP.3.1.eq.1}
\delta(P+Q)\leq{l+\sum_{i=0}^l\delta(\Gamma_{P,Q}(w^i))}\mbox{.}
\end{equation}

Now observe that $\Gamma_{P,Q}(w^i)$ has at most $f_0(Q)$ vertices. As a direct consequence, its diameter is at most $f_0(Q)-1$, and (\ref{Thm.DFP.3.1.eq.1}) yields
$$
\delta(P+Q)<(l+1)f_0(Q)\mbox{.}
$$

Since $l$ is the distance between two vertices in the graph of $P$, it is bounded above by $\delta(P)$, and we obtain the desired inequality.
\end{proof}

\section{The polytopes $\Xi(k,l)$ and $\tilde{\Xi}(k,l,m)$}\label{sec.DFP.4}

In this section, we describe two families of $3$-dimensional polytopes. The first family, which we will denote by $\Xi(k,l)$, shows that Theorem \ref{Thm.DFP.3.1} is sharp for the Minkoswki sum with a line segment, even when the diameter of the other summand is large. In other words, one can nearly double the diameter of a polytope by taking the Minkowski sum with a line segment. The other family, that will be denoted by $\tilde{\Xi}(k,l,m)$, will show that Theorem \ref{Thm.DFP.3.1} is also sharp for the Minkoswki sum with a polygon, even when both the number of vertices of the polygon and the diameter of the other summand are large.

Consider the $3$-dimensional polytope $\Xi(5,4)$ sketched in Fig. \ref{Fig.DFP.1}. The left of the figure shows $\Xi(5,4)$ from above, and the right of the figure shows it from below. The vertices colored blue are the vertices of a regular decagon $A$. In particular they all belong to $\mathbb{R}^2$, which we think of as a \emph{horizontal} plane. The red vertices are slightly above $\mathbb{R}^2$ and their orthogonal projection on $\mathbb{R}^2$ belongs to every other edge of the decagon. The green vertices are slightly below $\mathbb{R}^2$ and their orthogonal projection on $\mathbb{R}^2$ also belongs to every other edge of $A$, but with the requirement that a red and a green vertex never project on the same edge of $A$. It follows that $\Xi(5,4)$ has vertical facets, sketched in the center of the figure, each with two blue vertices and three other vertices, either all red or all green. The way the vertical facets are glued to the other facets of $\Xi(5,4)$ is indicated by arrows in the figure. The polytope $\Xi(5,4)$ also has two congruent horizontal facets colored grey in Fig. \ref{Fig.DFP.1}, each with $20$ vertices.
\begin{figure}
\begin{centering}
\includegraphics{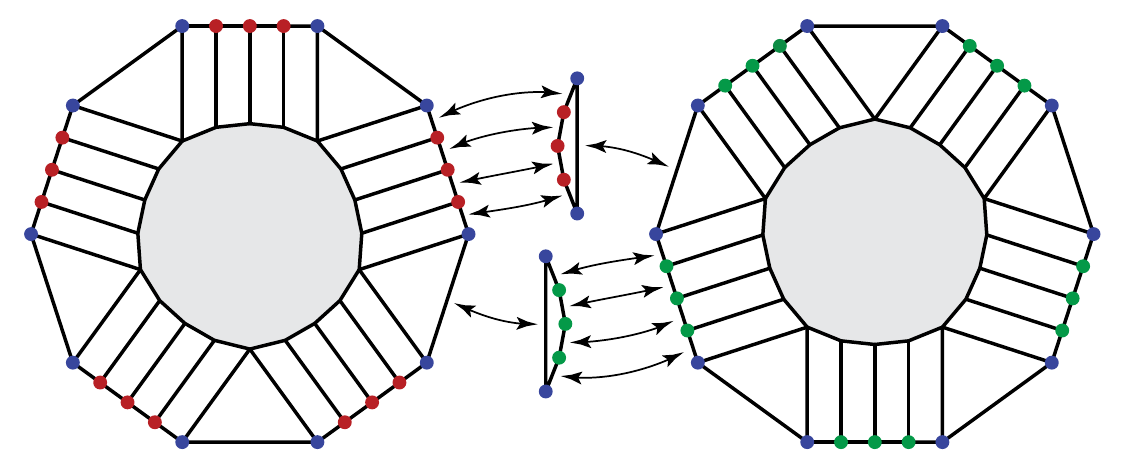}
\caption{The $3$-dimensional polytope $\Xi(5,4)$.}\label{Fig.DFP.1}
\end{centering}
\end{figure}
All the other facets of $\Xi(5,4)$ are either quadrilaterals or isoceles triangles. Each quadrilateral shares an edge with a horizontal facet and an edge with a vertical facet. Each triangle shares a vertex with a horizontal facet and an edge with a vertical facet. Observe that $\Xi(5,4)$ admits a natural generalization. One can define a similar $3$-dimensional polytope whose projection on $\mathbb{R}^2$ is a regular polygon with $2k$ vertices (which we shall also denote by $A$) instead of a decagon, and such that there are $l-1$ red (or green) vertices between two blue vertices, instead of just $3$. The resulting $3$-dimensional polytope, which we will denote by $\Xi(k,l)$, still has two horizontal facets, each with $kl$ vertices. It also has $2k$ vertical facets, each with two blue vertices and $l-1$ red or green vertices. The other facets of $\Xi(k,l)$ are $2k$ isoceles triangles and $2kl$ quadrilaterals. 

\begin{prop}\label{Prop.DFP.4.1}
The diameter of $\Xi(k,l)$ is at most $k+l+2$.
\end{prop}
\begin{proof}
Observe that the distance in the graph of $\Xi(k,l)$ from a red or green vertex to a blue vertex is at most $l/2$. Since the vertices of the horizontal facets are adjacent to a red or a green vertex, their distance to a blue vertex in the graph of $\Xi(k,l)$ is at most $l/2+1$. As two blue vertices are at distant by at most $k$ in the graph of $A$, we obtain the desired bound.
\end{proof}

\begin{prop}\label{Prop.DFP.4.2}
The Minkowski sum of $\Xi(k,4)$ with a vertical line segment has diameter at least $2k$.
\end{prop}
\begin{proof}
First observe that taking the Minkowski sum of $\Xi(k,4)$ with a vertical line segment $\Sigma$ does not modify the non-vertical facets of $\Xi(k,4)$, except for a possible translation. The only facets of $\Xi(k,4)$ whose geometry is modified by the Minkowski sum are the vertical ones. In these facets, the blue vertices are replaced by a translate of $Q$. The two vertices of this edge can be understood as two copies of a blue vertex, and will also be referred to as blue vertices.
\begin{figure}
\begin{centering}
\includegraphics{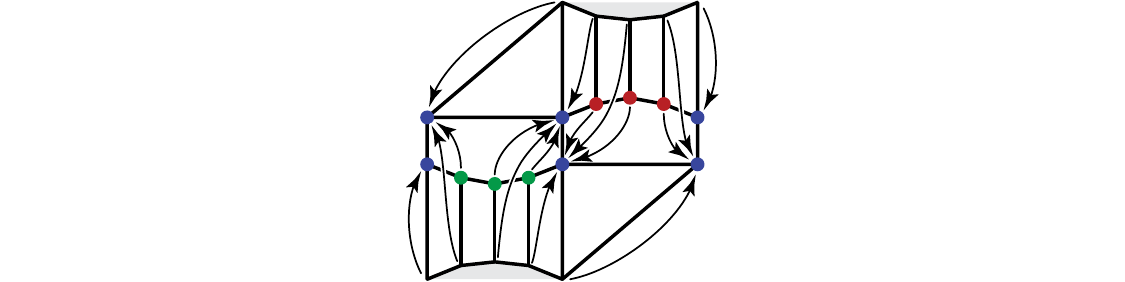}
\caption{The map $\lambda$.}\label{Fig.DFP.2}
\end{centering}
\end{figure}
In particular, the vertical facets of $\Xi(k,4)+\Sigma$ incident to a given blue vertex now share an edge, as shown on Fig. \ref{Fig.DFP.2}. Now consider the map $\lambda$ that sends each blue vertex of $\Xi(k,4)+\Sigma$ to itself and every other vertex of $\Xi(k,4)+\Sigma$ to a blue vertex, as indicated with arrows in Fig. \ref{Fig.DFP.2}. Note that the figure only depicts $\lambda$ next to a pair of vertical facets, but the rest of the map can be recovered using the rotational symmetry of $\Xi(k,4)+\Sigma$. Observe that $\lambda$ maps any two adjacent vertices of $\Xi(k,4)+\Sigma$ to adjacent or identical vertices. In particular, this map will transform a path between two blue vertices in the graph of $\Xi(k,4)+\Sigma$ into a path whose length has not increased between the same two blue vertices. Along the path resulting from the transformation, all the vertices are blue. As a consequence, the distance between two blue vertices can be measured within the cycle induced by blue vertices in the graph of $\Xi(k,4)+\Sigma$. Since this cycle has diameter $2k$, then $\Xi(k,4)$ has diameter at least $2k$.
\end{proof}

Combining Propositions \ref{Prop.DFP.4.1} and \ref{Prop.DFP.4.2} immediately shows that the upper bound provided by Theorem \ref{Thm.DFP.3.1} is asymptotically sharp for the Minkowski sum with a line segment, when the diameter of the other summand grows large.

\begin{thm}\label{Thm.DFP.4.1}
If $\Sigma$ is a vertical line segment, then
$$
\lim_{k\rightarrow\infty}\frac{\delta(\Xi(k,4)+\Sigma)}{\delta(\Xi(k,4))}=2\mbox{.}
$$
\end{thm}

The polytope $\Xi(k,l)$ is now modified into another polytope whose diameter gets multiplied by the number of vertices of a well-chose polygon (whose number of vertices is arbitrary) under the Minkowski sum with this polygon.
\begin{figure}[b]
\begin{centering}
\includegraphics{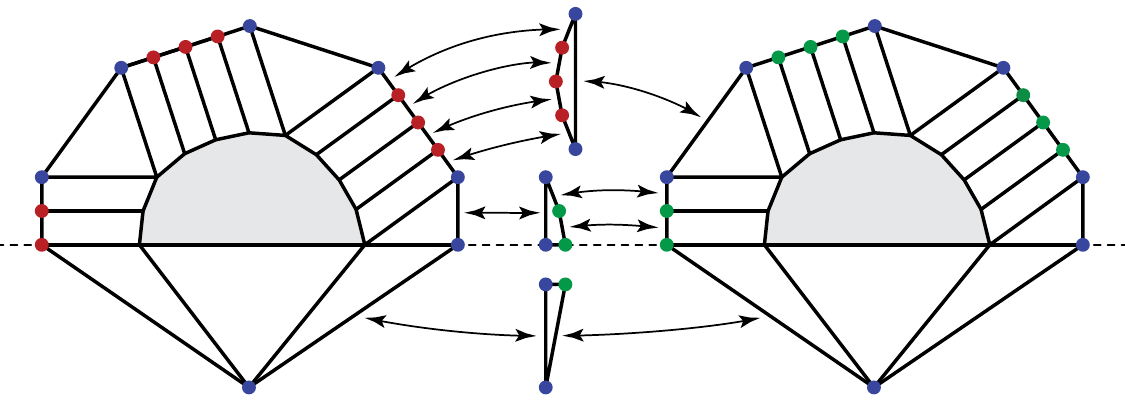}
\caption{The polytope $\Theta(5,4)$.}\label{Fig.DFP.3}
\end{centering}
\end{figure}
The first step of this modification, depicted in Fig. \ref{Fig.DFP.3} when $k=5$ and $l=4$, consists in cutting $\Xi(k,l)$ in half and replacing the removed half by a pyramid over an octogon. The cut is performed along a vertical plane $M$ that contains the center of two opposite edges of $A$. If $l$ is even, which we will assume from now on, then the intersection of $M$ and $\Xi(k,l)$ is an octogon whose vertices are two blue vertices, two red or green vertices, and two vertices of each grey facet. The plane $M$ is depicted as a dashed line in Fig. \ref{Fig.DFP.3} and shows the eight vertices of $M\cap\Xi(5,4)$. Note that, when $k$ is odd, $M\cap\Xi(k,l)$ has exactly one red vertex and one green vertex. When $k$ is even, $M\cap\Xi(k,l)$ has two red vertices or two green vertices depending on which pair of opposite edges of $A$ is cut in half by $M$. Now consider the polytope $\Theta(k,l)$ obtained by replacing the portion of $\Xi(k,l)$ on one side of $M$ by a pyramid over $M\cap\Xi(k,l)$, as shown on Fig. \ref{Fig.DFP.3} when $k=5$ and $l=4$. The apex $a$ of this pyramid, shown at the bottom of the figure, is placed in the horizontal plane $\mathbb{R}^2$ in such a way that the orthogonal projection of $a$ on $M$ is the center of $A$. The orthogonal projection on $\mathbb{R}^2$ of the resulting polytope is now a polygon with $k+3$ vertices. In the following $a$ will be thought of as a blue vertex. The way the vertical facets of $\Theta(k,l)$ are glued to the other facets of $\Theta(k,l)$ is indicated by arrows in Fig. \ref{Fig.DFP.3}. Note in particular that two vertical facets of $\Xi(k,l)$ have been cut in half in the process, and that $\Theta(k,l)$ has two new, right-angled vertical triangular facets incident to $a$.

We will further modify $\Theta(k,l)$ into a polytope $\tilde{\Xi}(k,l,m)$ by glueing small polytopes to the vertical facets that do not have a vertex in $M$. In order to build these polytopes, we will use homothetic translates of the vertical polygon $\Pi$ with $m+1$ vertices depicted on the left of Fig. \ref{Fig.DFP.4} when $m=4$. Let us first describe this polygon. The intersection of $\Pi$ with $M$ is the longest edge of $\Pi$, which we refer to by $e$. As shown on the figure, the vertices $\Pi$ outside of $M$ are not on the same side of $M$ than $a$, and their orthogonal projection on $M$ belongs to the relative interior of $e$.
\begin{figure}[b]
\begin{centering}
\includegraphics{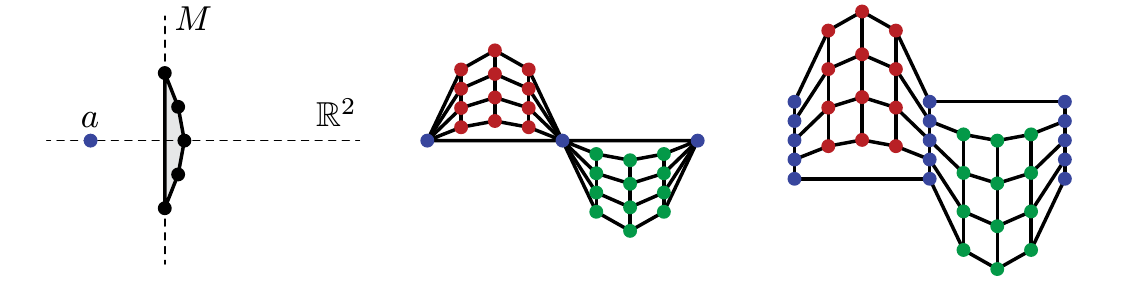}
\caption{The polygon $\Pi$ (left), the polytopes $P_F$ used to build $\tilde{\Xi}(k,l,m)$ from $\Theta(k,l)$ (center), and a sketch of the Minkoswki sum between $\Pi$ and these polytopes (right).}\label{Fig.DFP.4}
\end{centering}
\end{figure}
The largest distance to $M$ of a vertex of $\Pi$ will be denoted by $\varepsilon$. Note that $\varepsilon$ can be arbitrarily small which will be instrumental for the construction of $\tilde{\Xi}(k,l,m)$.

Now consider a vertical facet $F$ of $\Theta(k,l)$ that does not have a vertex in $M$. The announced polytope $P_F$, that we will glue to $F$, will be the convex hull of $F$ and of $l-1$ polygons homothetic to $\Pi$. Consider a red or a green vertex, say $v$, of $F$ and call $e'$ the vertical line segment incident to $v$ whose other vertex is in the horizontal edge of $F$. Denote by $\alpha$ the real number such that $\alpha{e}$ and $e'$ have the same length. We can then translate $\alpha\Pi$ and glue it to $F$ in such a way that $e$ and $e'$ coincide. The polytope $P_F$ is the convex hull of $F$ and of the $l-1$ homothetic translates of $\Pi$ glued to $F$ when $v$ ranges over the red or green vertices of $F$. The projection of $P_F$ back on $F$ is depicted in the center of Fig. \ref{Fig.DFP.4} for two consecutive vertical facets of $\Theta(k,4)$ when $m=5$. Note that the projection is made along the direction orthogonal to $M$. Further note that, apart from two blue vertices, all the vertices of $P_F$ will be colored red or green depending on whether $F$ has red or green vertices. If $\varepsilon$ is small enough, glueing these polytopes to each of the vertical facets of $\Theta(k,l)$ that do not have a vertex in $M$ results in a new polytope $\tilde{\Xi}(k,l,m)$ whose vertex set contains all the vertices of $\Theta(k,l)$, together with $(k-1)(l-1)(m-1)$ new vertices.

 \begin{prop}\label{Prop.DFP.4.3}
The diameter of $\tilde{\Xi}(k,l,m)$ is at most $(k+3)/2+l+2$.
\end{prop}
\begin{proof}
We proceed as in the proof of Proposition \ref{Prop.DFP.4.1}. Every vertex in the graph of $\tilde{\Xi}(k,l,m)$ is distant by at most $l/2+1$ of a blue vertex. Since there are $k+3$ blue vertices and these vertices induce a cycle in the graph of $\tilde{\Xi}(k,l,m)$, two of them are distant by at most $(k+3)/2$ in this graph. Therefore, we obtain an upper bound of $(k+3)/2+l+2$ on the diameter of $\tilde{\Xi}(k,l,m)$.
\end{proof}

By Lemma \ref{Lem.DFP.1.1}, when taking the Minkowski sum of $\tilde{\Xi}(k,l,m)$ with the polygon $\Pi$, the only faces whose geometry is affected are the vertical facets of $\tilde{\Xi}(k,l,m)$ and the faces of the polytopes $P_F$ for each of the vertical facets $F$ of $\Theta(k,l)$ that does not contain a vertex in $M$. Consider such a facet $F$ of $\Theta(k,l)$. By construction, each of the facets of $P_F$ is parallel to an edge of $\Pi$. In particular, according to Lemma \ref{Lem.DFP.1.1}, the Minkowski sum with $\Pi$ affects these facets as shown on the right of Fig. \ref{Fig.DFP.4}. Note that each of the blue vertices of $P_F$ will be copied $m+1$ times. Each of these copies will be thought of as a blue vertex. The two vertical facets of $\tilde{\Xi}(k,l,m)$ obtained by cutting in half a vertical facet of $\Xi(k,l,m)$ also each gain exactly $m$ new blue vertices. The vertical triangular facets of $\tilde{\Xi}(k,l,m)$ incident to $a$ are transformed into two quadrilaterals. In particular $a$ gives rise to two copies obtained from the Minkowski sum of $a$ with the edge $e$ of $\Pi$. These copies will both be considered blue vertices.

It follows that $\tilde{\Xi}(k,l,m)+\Pi$ has exactly $k(m+1)+4$ blue vertices that induce a cycle in the graph of $\tilde{\Xi}(k,l,m)+\Pi$. A portion of this cycle is depicted on Fig.~\ref{Fig.DFP.5}. When $l$ is large enough, an argument similar to the one used in the proof of Proposition \ref{Prop.DFP.4.2} will show that the long geodesics in the graph of $\tilde{\Xi}(k,l,m)+\Pi$ will mostly visit a sequence of blue vertices.

\begin{prop}\label{Prop.DFP.4.4}
If $l\geq2m+4$, then the Minkowski sum of $\tilde{\Xi}(k,l,m)$ with $\Pi$ has diameter at least $k(m+1)/2+1$.
\end{prop}
\begin{proof}
We will proceed in the same way as for Proposition \ref{Prop.DFP.4.2}. As already observed above, $\tilde{\Xi}(k,l,m)+\Pi$ has $k(m+1)+4$ blue vertices that induce a cycle in its graph. Hence, we only need to find a map $\lambda$ that takes each vertex of $\tilde{\Xi}(k,l,m)+\Pi$ to a blue vertex in such a way that two adjacent vertices are sent to either adjacent or identical blue vertices.

First consider the facets of $\tilde{\Xi}(k,l,m)+\Pi$ sketched on the right of Fig. \ref{Fig.DFP.4}. The way $\lambda$ affects the vertices of these facets is shown on the left and in the center of Fig. \ref{Fig.DFP.5}. As can be seen, the sketch has been deformed for clarity, which does not matter here since $\lambda$ is a combinatorial object. Observe that the red and green vertices are arranged in layers bounded by a blue vertex on the left and on the right. The number of red or green vertices in each of these layers is $l-1$. There is an additional layer made up of the two blue vertices of an horizontal edge of $\tilde{\Xi}(k,l,m)+\Pi$, shown below the red vertices and above the green vertices. The map $\lambda$ takes the first green or red vertex in a layer (from the left or from the right of the layer) to the blue vertex closest to it. The second green or red vertex in a layer will be sent to the blue vertex closest to it in the next layer and so on. Upon reaching the layer made up of a single vertical edge of $\tilde{\Xi}(k,l,m)+\Pi$, vertices will all be send to the vertex of this edge closest to them in the graph of $\tilde{\Xi}(k,l,m)+\Pi$. If $l\geq2m+4$, then $\lambda$ takes adjacent vertices to either adjacent or identical blue vertices, as desired. Note that, since $l$ is even, there is a vertex in the center of each layer. This vertex can be sent indifferently to one of the vertices of the horizontal edge of $\tilde{\Xi}(k,l,m)+\Pi$ in the last layer.

The map $\lambda$ is sketched on the right of Fig. \ref{Fig.DFP.5} for the vertical facets of $\tilde{\Xi}(k,l,m)$ obtained by cutting a facet of $\Xi(k,l)$ in half. The vertices that belong to $M$ are shown on the right of the figure.
\begin{figure}
\begin{centering}
\includegraphics{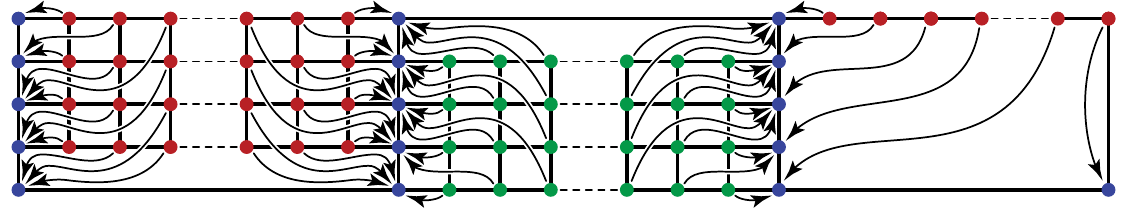}
\caption{The map $\lambda$ for the polytope $\tilde{\Xi}(k,l,m)$.}\label{Fig.DFP.5}
\end{centering}
\end{figure}
Note that $\lambda$ takes the red or green vertex in $M$ to the blue vertex in $M$. Further note that several vertices may be sent to the blue vertex shown on the bottom left of the facet in case $l$ grows large.

It remains to explain where $\lambda$ sends the vertices of the horizontal grey facets. This will be similar to what is shown in Fig. \ref{Fig.DFP.2}. The vertices that do not belong to $M$ will be sent to $\lambda(v)$ if they are adjacent to a red or a green vertex $v$ and to any one of the two blue vertex they are adjacent to otherwise. The four vertices that belong to $M$ will be sent to $a$ if they are adjacent to a red or green vertex and to the blue vertex in $M$ they are adjacent to otherwise.

This defines a map $\lambda$ such that sends any two adjacent vertices of $\tilde{\Xi}(k,l,m)+\Pi$ to either adjacent or identical blue vertices.
\end{proof}

We obtain the following by combining Propositions \ref{Prop.DFP.4.3} and \ref{Prop.DFP.4.4}.

\begin{thm}\label{Thm.DFP.4.2}
If $l\geq2m+4$, then
$$
\lim_{k\rightarrow\infty}\frac{\delta(\tilde{\Xi}(k,l,m)+\Pi)}{\delta(\tilde{\Xi}(k,l,m))}=m+1\mbox{.}
$$
\end{thm}

In other words, the Minkowski sum with $\Pi$ multiplies the diameter of $\tilde{\Xi}(k,l,m)$ by the number of vertices of $\Pi$, even though both of these quantities can grow arbitrarily large. This might come as a surprise. Indeed, while a geodesic in the graph of $\Pi$ never visits more than half of the vertices, the geodesics in the graph of $\tilde{\Xi}(k,l,m)+\Pi$ will visit an arbitrarily large number of copies of each vertex of $\Pi$. This proves that Theorem \ref{Thm.DFP.3.1} is sharp when the diameter of one summand is arbitrarily large, and the other summand is a line segment or an arbitrarily large polygon. Note that, by taking consecutive prisms over $\Xi(k,l)$ and $\tilde{\Xi}(k,l,m)$, one obtains that, for any fixed dimension $d$ greater than $2$, Theorem \ref{Thm.DFP.3.1} is sharp when one summand is $d$-dimensional and its diameter is arbitrarily large, while the other summand is a line segment or an arbitrarily large polygon. Further note that, when both summands have dimension at most $2$, the diameter of their Minkowski sum is better behaved since it is always at most, and can be equal to the sum of the diameters of the two summands.

This begs the question whether Theorem \ref{Thm.DFP.3.1} remains sharp when both summands are high dimensional. More precisely, we ask the following.

\begin{qtn}
Does there exist two polytopes $P$ and $Q$, both of dimension at least $3$ such that $\delta(P)$ and $f_0(Q)$ are arbitrarily large, while the ratio between $\delta(P+Q)$ and $\delta(P)f_0(Q)$ gets arbitrarily close to $1$?
\end{qtn}

\noindent{\bf Acknowledgements.} The authors thank Komei Fukuda for inspiring comments and insights that nurtured this work from the beginning.

\bibliography{MinkowskiDiameters}
\bibliographystyle{ijmart}

\end{document}